\documentclass[12pt]{article}
\usepackage{amsmath,amsthm}
\usepackage{amsfonts,amssymb}
\usepackage{graphicx,epstopdf}
\usepackage[usenames,dvipsnames]{color}

\newtheorem{theorem}{Theorem}

\newtheorem{lemma}[theorem]{Lemma}

\DeclareMathAlphabet{\mathpzc}{OT1}{pzc}{m}{it}

\def\A{{\boldsymbol A}}
\def\C{{\boldsymbol C}}
\def\D{{\boldsymbol D}}
\def\I{{\boldsymbol I}}
\def\J{{\boldsymbol J}}

\def\0{{\boldsymbol 0}}

\def\FF{{\bf F}}

\def\R{\mathbb{R}}
\def\C{\mathbb{C}}

\def\rD{{\rm (D)}}
\def\rJ{{\rm (J)}}

\def\tm{{\widetilde \mu}}

\def\tz{{\widetilde z}}

\DeclareMathOperator{\erf}{erf}

\begin{document}

\title{ADI schemes for valuing European options under the Bates model}

\author{
Karel J. in 't Hout\footnote{
Department of Mathematics and Computer Science, University of Antwerp,
Middelheimlaan 1, B-2020 Antwerp, Belgium.
Email: karel.inthout@uantwerp.be}
~and\,
Jari Toivanen\footnote{
Faculty of Information Technology, PO Box 35, FI-40014 University of 
Jyv\"{a}skyl\"{a}, Finland.
Email: jari.a.toivanen@jyu.fi}
}

\date{}

\maketitle

\begin{abstract}

This paper is concerned with the adaptation of alternating direction implicit
(ADI) time discretization schemes for the numerical solution of partial
integro-differential equations (PIDEs) with application to the Bates model
in finance.
Three different adaptations are formulated and their (von Neumann) stability 
is analyzed.
Ample numerical experiments are provided for the Bates PIDE, 
illustrating the actual stability and convergence behaviour of the three
adaptations.

\medskip\noindent
{\it Keywords:}~ partial integro-differential equations, operator splitting
methods, alternating direction implicit schemes, stability, Bates model.
\end{abstract}

\section{Introduction}
The traditional asset price model in financial option valuation theory is 
the geometric Brownian motion. 
It is well-known that this model assumes constant volatility while the market 
prices of options clearly indicate varying volatility. 
For the valuation of options with long maturities, stochastic volatility models 
like the Heston stochastic volatility model \cite{Heston93} are a common means 
to introduce such variability. 
For short maturities, however, pure Brownian motion based models 
such as the Heston model often require 
excessively large volatilities to explain the market prices of options. 
A modern way to resolve this is to incorporate jumps in the asset price model,
like the classical Merton jump-diffusion model \cite{Merton76} does. 
The Bates model \cite{Bates96} combines the Merton model and the Heston model. 
This asset price model includes both jumps and stochastic volatility and it is
therefore popular for valuing options with short as well as long maturities.

Under the Bates model, a two-dimensional parabolic partial integro-differential 
equation (PIDE) can be derived for the values of European-style options with the 
spatial variables representing the underlying asset price and its instantaneous 
variance \cite{Cont04}.
The differential operator is of the convection-diffusion type. 
The integral operator, which stems from the jumps, couples all option values 
in the asset price direction. 

This paper deals with the numerical solution of the Bates PIDE.
We follow the method of lines approach, where the PIDE is first discretized 
in the spatial variables and the resulting semidiscrete system of ordinary 
differential equations (ODEs) is solved by applying a suitable implicit time 
discretization scheme.
Due to the nonlocal nature of the integral operator, semidiscretization 
leads to a large, dense system matrix.
For an implicit time discretization scheme one thus faces two challenges: 
the efficient treatment of the two-dimensional convection-diffusion part 
and the efficient treatment of the dense integral part. 
In this paper we focus on operator splitting methods, which handle these 
two parts separately in an effective manner.
The following reviews methods proposed in the literature.

For the time discretization of the two-dimensional partial differential 
equation (PDE) arising from the Heston model, various {\it alternating 
direction implicit (ADI)} schemes were studied in \cite{intHout10}. 
ADI and other directional splitting methods, see e.g.~\cite{Ikonen08},
are highly efficient. 
In each time step they yield a sequence of unidirectional linear systems, 
which can be solved very fast using LU factorization.
Since the original papers \cite{Douglas55,Douglas56,Peaceman55}, the literature 
on ADI schemes for time-dependent convection-diffusion equations is extensive. 
In financial option valuation, correlations between the underlying 
stochastic processes (asset price, variance) give rise to mixed spatial-derivative 
terms in the pertinent PDEs.
First-order ADI schemes have been investigated for such PDEs in \cite{McKee70,McKee96}
and second-order ADI schemes in \cite{Craig88,Haentjens12,intHout10,intHout17b,intHout11,
intHout13,intHout07,intHout09b,Mishra16}, where the mixed derivative part is always 
treated in an explicit fashion.

Under the Merton model a one-dimensional PIDE for European option values holds. 
Application of a classical implicit time discretization scheme such as 
Crank--Nicolson or BDF2 leads in every time step to the solution of a linear 
system with a dense matrix. 
An iterative method requiring only multiplications by a dense matrix and 
solutions to tridiagonal systems can often solve these linear systems with 
a few iterations \cite{Almendral05,dHalluin05,Salmi11,Tavella00}.
Alternatively, a splitting method for the time discretization of the Merton 
PIDE was proposed in \cite{Andersen00}.
This method, of the Peaceman--Rachford type, treats the integral part in an
implicit manner and employs FFT to reduce computational cost.
An {\it implicit-explicit (IMEX)} method that treats the sparse diffusion
part implicitly and the dense integral part explicitly was investigated in 
\cite{Cont05}.
The pertinent method is only first-order. 
Second-order IMEX methods such as the IMEX midpoint and IMEX--CNAB 
(Crank--Nicolson Adams--Bashforth) schemes were subsequently
investigated in \cite{Feng08,Kwon11a,Salmi14b} and IMEX schemes of the 
Runge--Kutta type were examined in \cite{Briani07}.

The IMEX--CNAB method has been applied to the Bates PIDE in \cite{Salmi14a}. 
For the solution of the linear system in each time step, both LU 
factorization and an algebraic multigrid method were considered.
An adaptive finite difference discretization combined with the
IMEX--CNAB method and LU factorization was constructed in 
\cite{vonSydow15}.
In \cite{Ballestra16,Ballestra10} a first-order directional splitting 
method was used together with Richardson extrapolation with the aim to 
obtain second-order accuracy.
Directional splitting methods of the Strang type were studied in 
\cite{Chiarella09,Toivanen10a}.
In \cite{Itkin11} the authors employed Strang splitting between the 
convection-diffusion and integrals parts and applied a second-order 
ADI scheme to solve the two-dimensional convection-diffusion substeps 
in this approach; see also \cite{Itkin17}.
A novel, direct adaptation of ADI schemes to two-dimensional PIDEs has 
recently been proposed in \cite{Kaushansky17}.
Here the integral and mixed derivative parts are treated jointly in an 
explicit fashion.

For more information on operator splitting methods like ADI and IMEX
schemes in finance we refer to \cite{intHout17a,intHout16}.
A good general reference is \cite{Hundsdorfer03}.

In this paper we consider the adaptation of ADI schemes for the time
discretization of two-dimensional PIDEs with the Bates model as the
leading example.
We formulate three types of adaptations: two of these are novel and 
one agrees with that from \cite{Kaushansky17}.
The obtained operator splitting methods are all of second-order and
treat both the integral and mixed derivative parts in an explicit
manner.
At each time step they require only one or two multiplications 
with the dense matrix corresponding to the integral part.
Hence, the resulting methods are computationally very efficient.

An outline of this paper is as follows. 
Section~\ref{BatesSec} describes the Bates model, the PIDE for 
European option values, and its spatial discretization.  
In Section~\ref{ADISec} the three adaptations of ADI schemes to 
PIDEs are formulated and their stability is analyzed.
Numerical experiments for these adaptations applied to the Bates
PIDE are presented in Section~\ref{ExampleSec}. 
The conclusions end the paper in Section~\ref{ConcSec}.

\setcounter{equation}{0}
\setcounter{theorem}{0}
\section{Bates model and its semidiscretization}\label{BatesSec}
We consider in this paper the Bates model \cite{Bates96}, which merges the Heston
stochastic volatility model \cite{Heston93} and the Merton jump-diffusion
model \cite{Merton76} for the asset price $S_\tau \ge 0$.
Here the instantaneous variance $V_\tau \ge 0$ follows a CIR process
with mean-reversion level \mbox{$\eta > 0$}, mean-reversion rate $\kappa > 0$
and volatility-of-variance $\sigma > 0$. The underlying Brownian motions for $S_\tau$ 
and $V_\tau$ have correlation $\rho \in [-1,1]$. The jump process is a compound 
Poisson process with intensity $\lambda > 0$. The probability of a jump from 
$S_\tau$ to $S_\tau Y$ with random variable $Y>0$ is defined by the log-normal 
probability density function
\begin{equation*}
f(y)=\frac{1}{y \delta\sqrt{2\pi}}e^{-\frac{(\log y-\gamma)^2}{2\delta^2}} \quad (y>0),
\end{equation*}
where $\gamma$ and $\delta$ are given real constants with $\delta>0$ that are equal 
to the mean and standard deviation, respectively, of the normal random variable $\log (Y)$.
Let $r\ge 0$ be the risk-free interest rate and let $\varepsilon$ denote the 
mean relative jump size, 
given by $\varepsilon=e^{\gamma+\delta^2/2}-1$.
Consider an European-style option with maturity time $T$.
Under the Bates model, if $S_{T-t}=s$ and $V_{T-t}=v$, then the option value $u(s,v,t)$ 
satisfies the PIDE
\begin{equation}\label{L1}
\begin{split}
\frac{\partial u}{\partial t} &{} =
\frac{1}{2} s^2 v \frac{\partial^2 u}{\partial s^2}
+ \rho \sigma s v \frac{\partial^2 u}{\partial s \partial v}
+ \frac{1}{2} \sigma^2 v \frac{\partial ^2u}{\partial v^2}
+ (r - \lambda \varepsilon) s \frac{\partial u}{\partial s}
+ \kappa ( \eta - v ) \frac{\partial u}{\partial v} \\
&{} - (r + \lambda) u
+ \lambda \int_{0}^\infty u(sy,v,t) f(y) dy
\end{split}
\end{equation}
whenever $s>0$, $v>0$, $0 < t \le T$.
The initial condition for \eqref{L1} is
\begin{equation}\label{initial}
u(s,v,0) = \phi(s),
\end{equation}
where $\phi$ denotes the pay-off function, which specifies the value of the option at
expiry. In this paper we shall consider vanilla put options.
The pay-off function is then $\phi(s) = \max \{ K - s,\, 0 \}$ with given
strike price $K$.

For feasibility of the numerical solution, the unbounded spatial domain 
is truncated 
to $[0,\, S_{\max}] \times [0,\, V_{\max}]$ with sufficiently large $S_{\max}$ and 
$V_{\max}$ chosen such that the error caused by this truncation is negligible. 
We pose the Dirichlet and Neumann boundary conditions
\begin{equation}\label{BCs}
\begin{split}
u(0,v,t) &= e^{-rt} K \phantom{0}~~ (0 \le v \le V_{\max},\; 0 < t \le T), \\
\phantom{\frac{\partial u}{\partial s}} u(S_{\max},v,t) &= 0 \phantom{e^{-rt} K}~~
(0 \le v \le V_{\max},\; 0 < t \le T), \\
\frac{\partial u}{\partial v} (s,V_{\max},t) &= 0 \phantom{e^{-rt} K}~~
(0 \le s \le S_{\max},\; 0 < t \le T),
\end{split}
\end{equation}
and at the boundary $v=0$ it is assumed that the PIDE \eqref{L1} holds, compare 
\cite{Ekstrom10}.

For the spatial discretization on the computational domain $[0,\, S_{\max}] \times [0,\, V_{\max}]$
a smooth, nonuniform Cartesian grid is chosen with meshes in the $s$- and $v$-directions of the type 
considered in \cite{Haentjens12}. 
This grid has a large fraction of points in the neighbourhood of the important location $(s,v)=(K,0)$. 
Let the grid points in the $s$-direction be denoted by $0 = s_0 < s_1 < \cdots < s_{m_1} = S_{\max}$
and the grid points in the $v$-direction by $0 = v_0 < v_1 < \cdots < v_{m_2} = V_{\max}$. 
Let the mesh widths $\Delta s_i = s_i - s_{i-1}$, $\Delta v_j = v_j - v_{j-1}$ and write $u_{i,j}$ 
for the value of $u$ at grid point $(s_i,v_j)$.
For the partial derivatives of $u$ with respect to $s$ we apply the standard central finite
differences
\begin{equation}\label{cfd1}
\frac{\partial u_{i,j}}{\partial s}
\approx
\frac{-\Delta s_{i+1}}{\Delta s_i (\Delta s_i + \Delta s_{i+1})}\, u_{i-1,j}
+ \frac{\Delta s_{i+1} - \Delta s_i}{\Delta s_i \Delta s_{i+1}}\, u_{i,j}
+ \frac{\Delta s_i}{(\Delta s_i + \Delta s_{i+1}) \Delta s_{i+1}}\, u_{i+1,j}
\end{equation}
and
\begin{equation}\label{cfd2}
\frac{\partial^2 u_{i,j}}{\partial s^2}
\approx
\frac{2}{\Delta s_i (\Delta s_i + \Delta s_{i+1})}\, u_{i-1,j}
- \frac{2}{\Delta s_i \Delta s_{i+1}}\, u_{i,j}
+ \frac{2}{(\Delta s_i + \Delta s_{i+1}) \Delta s_{i+1}}\, u_{i+1,j}\,.
\end{equation}
For the partial derivatives of $u$ with respect to $v$ the analogous finite differences
are considered, with the exception of the boundary $v=0$, where we apply the forward formula
\begin{equation}\label{ufd}
\frac{\partial u_{i,0}}{\partial v}
\approx
-\frac{2{\Delta v}_{1}+{\Delta v}_{2}}{{\Delta v}_{1}({\Delta v}_{1}+{\Delta v}_{2})}\, u_{i,0}
+ \frac{{\Delta v}_{1}+{\Delta v}_{2}}{{\Delta v}_{1}{\Delta v}_{2}}\, u_{i,1}
- \frac{{\Delta v}_{1}}{({\Delta v}_{1}+{\Delta v}_{2}){\Delta v}_{2}}\, u_{i,2}\,.
\end{equation}
For the mixed derivative $\partial^2 u/\partial s\partial v$ the central nine-point formula 
is used that is obtained by successive application of the standard central finite differences 
for the first derivative in the $s$- and $v$-directions.

For the discretization of the integral term in \eqref{L1}, we employ linear 
interpolation for $u$ between consecutive grid points in the $s$-direction,
and $u(s,v,t)$ is approximated to be zero whenever $s\ge S_{\max}$. 
After some calculations, this leads to the quadrature rule
\begin{equation}
\begin{split}
& \int_{0}^\infty u(s_i y,v_j,\cdot) f(y) dy \\
& \approx \sum_{k=0}^{m_1-1}
\frac{1}{2 \Delta s_{k+1}} \left[ 
\left( c_{1,i,k} s_i 
- c_{0,i,k} s_{k+1} \right) u_{k,j}
+ \left( c_{0,i,k} s_k 
- c_{1,i,k} s_i \right) u_{k+1,j}
\right],
\end{split}
\end{equation}
where
\begin{equation}
c_{d,i,k} =
e^{d(\gamma + \delta^2/2)} \left[
\erf \left(\frac{\gamma + d\delta^2 - \log (s_{k+1}/s_i)}{\delta \sqrt{2}}\right) -
\erf \left(\frac{\gamma + d\delta^2 - \log (s_k/s_i)}{\delta \sqrt{2}}\right)
\right]
\end{equation}
for $d=0, 1$ with $\erf(\cdot)$ denoting the error function.
This same approximation of the integral term was used in \cite{Salmi11}, 
for example.
Here we adopt the convention $\log(0) = -\infty$ and $\erf(\infty)=1$.

Collecting the semidiscrete approximations to the grid point values 
$u(s_i,v_j,t)$ for $1\le i\le m_1-1$, $0\le j\le m_2$ into a vector $U(t)$ gives 
rise to a system of ODEs of 
the form
\begin{equation}\label{ODEsystem}
\dot{U}(t) = \A U(t) + G(t),~~0 < t \le T.
\end{equation}
Here $\A$ is a given $m\times m$-matrix and $G(t)$ is a given $m$-vector
with $m=(m_1-1)(m_2+1)$.
The vector $G(t)$ depends on the Dirichlet boundary data in \eqref{BCs}.
Let $\I$ denote the $m\times m$-identity matrix.
The matrix $\A$ can be written as 
\begin{equation}\label{ODE}
\A = \D - (r+\lambda) \I + \lambda \J,
\end{equation}
where $\D$ represents the convection-diffusion part of \eqref{L1} and $\J$ 
represents the integral. 
The matrix $\D$ is sparse, whereas $\J$ is a block diagonal matrix
with $m_2+1$ dense diagonal blocks of size $(m_1-1)\times (m_1-1)$.
It can be shown that all entries of $\J$ are nonnegative and all its row 
sums are bounded from above by one. In particular this implies that 
the spectrum of $\J$ belongs to the complex unit disk.

The ODE system \eqref{ODEsystem} is complemented with an initial vector $U(0)$
that is given by the values of the pay-off function $\phi$ at the spatial grid 
points where,
in view of the nonsmoothness of $\phi$ at the strike $K$, the values 
at the grid points $(s_i,v_j)$ with $s_i$ the point lying closest to $K$ are 
replaced by the cell average value
\begin{equation}\label{cellaver}
\frac{1}{\Delta s_{i+1/2}}
\int_{s_{i-1/2}}^{s_{i+1/2}} \phi(s)\, ds
\end{equation}
with
$
s_{i-1/2} = \tfrac{1}{2}(s_{i-1}+s_i),
s_{i+1/2} = \tfrac{1}{2}(s_i+s_{i+1}),
\Delta s_{i+1/2} = s_{i+1/2}-s_{i-1/2}.
$

\setcounter{equation}{0}
\setcounter{theorem}{0}
\section{Operator splitting methods}\label{ADISec}
\subsection{Adaptation of ADI schemes to PIDEs}\label{adaptations}
For the efficient temporal discretization of the semidiscrete system \eqref{ODEsystem}
we study in this paper operator splitting methods.
To this purpose, the matrix $\A$ is decomposed into four matrices,
\begin{equation}\label{splitA}
\A = \A_0^\rJ + \A_0^\rD + \A_1 + \A_2.
\end{equation}
Here $\A_1$, respectively $\A_2$, corresponds to the finite difference discretization 
of all spatial derivatives in the $s$-direction, respectively $v$-direction, minus half 
of the $(r+\lambda) \I$ matrix.
Next, $\A_0^\rD$ represents the finite difference discretization of the mixed derivative 
term in \eqref{L1} and $\A_0^\rJ = \lambda \J$ represents the finite difference discretization 
of the integral term.
We note that incorporating $(r+\lambda) \I$ in the above manner turns out to be beneficial 
for stability.
The vector $G(t)$, containing the boundary contributions, is decomposed analogously to $\A$.
It is convenient to define for $0\le t\le T$ and $V\in\R^m$,
\begin{equation*}
\begin{split}
\FF_0^\rJ (t,V) &= \A_0^\rJ V + G_0^\rJ (t),\\\\
\FF_0^\rD (t,V) &= \A_0^\rD V + G_0^\rD (t),\\\\
\FF_1 (t,V) &= \A_1 V + G_1 (t),\\\\
\FF_2 (t,V) &= \A_2 V + G_2 (t),
\end{split}
\end{equation*}
and write $\FF = \FF_0 + \FF_1 + \FF_2$, $\FF_0 = \FF_0^\rJ + \FF_0^\rD$
and $\FF^\rD = \FF_0^\rD + \FF_1 + \FF_2$.

The basis for the operator splitting approach under investigation in this paper is the 
class of ADI schemes, which is well-established in the
literature for option valuation PDEs, without integral term.
As a prominent representative of this class we select the {\it Modified Craig--Sneyd (MCS) 
scheme}, which was introduced in \cite{intHout09b}.

Let $\theta>0$ be a given parameter. 
Let step size $\Delta t= T/N$ with given integer $N\ge 1$ and temporal grid points $t_n = n \Delta t$ 
for $n=0,1,\ldots,N$.
If the integral term is absent, then $\FF_0 = \FF_0^\rD$ and the MCS scheme generates successively, 
in a one-step manner, approximations $U_n$ to $U(t_n)$ for $n=1,2,\ldots,N$ by
\begin{equation}\label{MCS1}
\left\{\begin{array}{lll}
Y_0 = U_{n-1}+\Delta t\, \FF(t_{n-1},U_{n-1}), \\\\
Y_j = Y_{j-1}+\theta\Delta t\, (\FF_j(t_n,Y_j)-\FF_j(t_{n-1},U_{n-1}))
\quad (j=1,2), \\\\
\widehat{Y}_0 = Y_0+\theta\Delta t\, (\FF_0(t_n,Y_2)-\FF_0(t_{n-1},U_{n-1})),\\\\
\widetilde{Y}_0 = \widehat{Y}_0+(\frac{1}{2}-\theta)\Delta t\, (\FF(t_n,Y_2)-\FF(t_{n-1},U_{n-1})), \\\\
\widetilde{Y}_j = \widetilde{Y}_{j-1}+\theta\Delta t\, (\FF_j(t_n,\widetilde{Y}_j)-\FF_j(t_{n-1},U_{n-1}))
\quad (j=1,2), \\\\
U_n = \widetilde{Y}_2.
\end{array}\right.
\end{equation}
The MCS scheme has classical order of consistency (that is, for fixed nonstiff ODEs) equal to
two for any parameter value $\theta$.
The particular choice $\theta= \frac{1}{2}$ yields the Craig--Sneyd (CS) scheme, 
introduced in \cite{Craig88}.

In the MCS scheme the $\FF_0$ part is treated in an explicit fashion and the  
$\FF_1$ and $\FF_2$ parts are treated in an implicit fashion.
In each step of the scheme four linear systems need to be solved, involving the two matrices
$\I - \theta \Delta t \A_j$ for $j=1,2$.
Since these two matrices are essentially tridiagonal (up to permutation), 
the linear systems can be solved 
very efficiently by using $LU$ factorizations, which can be computed upfront.
Notice further that $\widehat{Y}_0$ and $\widetilde{Y}_0$ can be combined in 
\eqref{MCS1}, so that just two evaluations of $\FF_0$ are required per time step.

We study three adaptations of the MCS scheme to PIDEs.
The {\it first adaptation}\, is straightforward. It consists of 
\eqref{MCS1} with $\FF_0 = \FF_0^\rJ + \FF_0^\rD$ as defined above.
The integral term is then conveniently treated in an explicit way, along with the mixed
derivative term.
This type of adaptation has recently been proposed in \cite{Kaushansky17}, where the (related)
Hundsdorfer--Verwer ADI scheme, instead of the MCS scheme, was considered.

The {\it second adaptation}\, of the MCS scheme to PIDEs is obtained by treating the integral term 
in a one-step, explicit, second-order fashion directly at the beginning of each time step:
\begin{equation}\label{MCS2}
\left\{\begin{array}{lll}
X_0 = U_{n-1}+\Delta t\, \FF(t_{n-1},U_{n-1}), \\\\
Y_0 = X_0 + \frac{1}{2}\Delta t\,(\FF_0^\rJ(t_n,X_0)-\FF_0^\rJ(t_{n-1},U_{n-1})),\\\\
Y_j = Y_{j-1}+\theta\Delta t\, (\FF_j(t_n,Y_j)-\FF_j(t_{n-1},U_{n-1}))
\quad (j=1,2), \\\\
\widehat{Y}_0 = Y_0+\theta\Delta t\, (\FF_0^\rD(t_n,Y_2)-\FF_0^\rD(t_{n-1},U_{n-1})),\\\\
\widetilde{Y}_0 = \widehat{Y}_0+(\frac{1}{2}-\theta)\Delta t\, (\FF^\rD(t_n,Y_2)-\FF^\rD(t_{n-1},U_{n-1})), \\\\
\widetilde{Y}_j = \widetilde{Y}_{j-1}+\theta\Delta t\, (\FF_j(t_n,\widetilde{Y}_j)-\FF_j(t_{n-1},U_{n-1}))
\quad (j=1,2), \\\\
U_n = \widetilde{Y}_2.
\end{array}\right.
\end{equation}
We mention that the scheme defined by the first three lines of \eqref{MCS2} 
followed by $U_n=Y_2$ has recently been studied, in a different context, in \cite{Arraras17} 
for problems without integral term.

In the {\it third adaptation}\, of the MCS scheme to PIDEs, the integral term is treated in a two-step, 
explicit, second-order fashion directly at the start of each time step:
\begin{equation}\label{MCS3}
\left\{\begin{array}{lll}
X_0 = U_{n-1}+\Delta t\, \FF^\rD(t_{n-1},U_{n-1}), \\\\
Y_0 = X_0 + \frac{3}{2}\Delta t\,\FF_0^\rJ(t_{n-1},U_{n-1}) - \frac{1}{2}\Delta t\,\FF_0^\rJ(t_{n-2},U_{n-2}),\\\\
Y_j = Y_{j-1}+\theta\Delta t\, (\FF_j(t_n,Y_j)-\FF_j(t_{n-1},U_{n-1}))
\quad (j=1,2), \\\\
\widehat{Y}_0 = Y_0+\theta\Delta t\, (\FF_0^\rD(t_n,Y_2)-\FF_0^\rD(t_{n-1},U_{n-1})),\\\\
\widetilde{Y}_0 = \widehat{Y}_0+(\frac{1}{2}-\theta)\Delta t\, (\FF^\rD(t_n,Y_2)-\FF^\rD(t_{n-1},U_{n-1})), \\\\
\widetilde{Y}_j = \widetilde{Y}_{j-1}+\theta\Delta t\, (\FF_j(t_n,\widetilde{Y}_j)-\FF_j(t_{n-1},U_{n-1}))
\quad (j=1,2), \\\\
U_n = \widetilde{Y}_2.
\end{array}\right.
\end{equation}
For starting \eqref{MCS3}, we define $U_1$ by \eqref{MCS1}.

It can be shown that each of the adaptations \eqref{MCS1}, \eqref{MCS2}, \eqref{MCS3} of the MCS 
scheme retains classical order of consistency equal to two for any value $\theta$.
The mixed derivative and integral terms are always handled in an explicit manner.
In \eqref{MCS1} and \eqref{MCS2} the integral term is dealt with in an explicit trapezoidal rule 
fashion, whereas in \eqref{MCS3} it is treated in a two-step Adams--Bashforth fashion.
An advantage of the adaptation \eqref{MCS3} is that only one multiplication with the dense matrix 
$\A_0^\rJ$ is required per time step, instead of two for \eqref{MCS1} and \eqref{MCS2}.
Treating the integral term in a two-step Adams--Bashforth fashion has recently been advocated in 
\cite{Salmi14b,Salmi14a} where IMEX multistep schemes, without directional splitting, were 
considered for option valuation PIDEs.

In the following section we analyze the stability of the three adaptations of the MCS 
scheme.

\subsection{Stability analysis}\label{StabAnal}
Let $\lambda_0$, $\mu_0$, $\mu_1$, $\mu_2$ denote any given complex numbers representing eigenvalues 
of the matrices $\A_0^\rJ$, $\A_0^\rD$, $\A_1$, $\A_2$, respectively.
For the stability analysis of the adaptations of the MCS scheme to PIDEs formulated in Section
\ref{adaptations} we consider the linear scalar test equation
\begin{equation}\label{testeqn}
\dot{U}(t) = (\lambda_0+\mu_0+\mu_1+\mu_2)\, U(t).
\end{equation}
Let $w_0 = \lambda_0\, \Delta t$ and $z_j = \mu_j\, \Delta t$ for $j=0,1,2$.
Application of the three adaptations \eqref{MCS1}, \eqref{MCS2}, \eqref{MCS3} to test equation 
\eqref{testeqn} yields the linear recurrence relations
\begin{equation}\label{rR}
U_n = R(w_0,z_0,z_1,z_2)\, U_{n-1}\,,
\end{equation}
\begin{equation}\label{rS}
U_n = S(w_0,z_0,z_1,z_2)\, U_{n-1}\,,
\end{equation}
and
\begin{equation}\label{rT}
U_n = T_1(w_0,z_0,z_1,z_2)\, U_{n-1} + T_0(w_0,z_0,z_1,z_2)\, U_{n-2}\,,
\end{equation}
respectively, with
\begin{eqnarray}
R(w_0,z_0,z_1,z_2)&=&\label{R}
1+\frac{z}{p}+\theta\frac{( w_0+z_0)z}{p^2}+(\tfrac{1}{2}-\theta)\frac{z^2}{p^2}\,,\\
S(w_0,z_0,z_1,z_2)&=&\label{S}
1+( 1+\tfrac{1}{2}w_0)\left(\frac{z}{p}+\theta\frac{z_0 z}{p^2}+(\tfrac{1}{2}-\theta)\frac{(z_0+z_1+z_2)z}{p^2}\right),~~~~~\\
T_1(w_0,z_0,z_1,z_2)&=&\label{T1}
1+(z+\tfrac{1}{2}w_0)\left(\frac{1}{p}+\theta\frac{z_0}{p^2}+(\tfrac{1}{2}-\theta)\frac{z_0+z_1+z_2}{p^2}\right),\\
T_0(w_0,z_0,z_1,z_2)&=&\label{T2}
-\tfrac{1}{2}w_0\left(\frac{1}{p}+\theta\frac{z_0}{p^2}+(\tfrac{1}{2}-\theta)\frac{z_0+z_1+z_2}{p^2}\right),
\end{eqnarray}
where
\begin{equation*}\label{zp}
z = w_0+z_0+z_1+z_2 
\quad\text{and}\quad 
p = (1-\theta z_1)(1-\theta z_2).
\end{equation*}
Clearly, if $w_0=0$, then $R$, $S$, $T_1$ reduce to the same rational function and $T_0$ vanishes.

We investigate the stability of the three processes \eqref{rR}, \eqref{rS}, \eqref{rT}.
By definition of $\A_j$ it holds for $j=1,2$ that
\[
\mu_j = \tm_j - \tfrac{1}{2}(r+\lambda),
\]
where $\tm_j$ represents an eigenvalue of the matrix corresponding to the finite difference 
discretization of all spatial derivatives in the $j$-th spatial direction.
Let $\tz_j = \tm_j\, \Delta t$ for $j=1,2$.
In the stability analysis of ADI schemes for two-dimensional convection-diffusion equations with 
mixed derivative term, the following condition on $z_0$, $\tz_1$, $\tz_2$ plays a key role:
\begin{equation}\label{cond1}
|z_0|\le 2\sqrt{\Re \tz_1\Re \tz_2}
~~{\rm and}~~
\Re \tz_1\le 0,~\Re \tz_2\le 0,
\end{equation}
where $\Re$ designates the real part.
It is satisfied \cite{intHout07} in the model (von Neumann) framework, with semidiscretization by 
second-order central finite differences on uniform, Cartesian grids and periodic boundary condition.
Consider the following, related condition on $w_0$, $z_0$, $z_1$, $z_2$\,:
\begin{equation}\label{cond2}
|z_0|+|w_0|\le 2\sqrt{\Re z_1\Re z_2}
~~{\rm and}~~
\Re z_1\le -\tfrac{1}{2} |w_0|,~\Re z_2\le -\tfrac{1}{2} |w_0|.
\end{equation}
For the integral term there holds 
\[
\lambda_0 = \lambda\nu,
\] 
where $\nu\in\C$ represents an eigenvalue of the matrix $\J$.
A natural assumption on $\nu$ is
\[
|\nu|\le 1,
\]
compare Section \ref{BatesSec}.
For $w_0$, $z_0$, $z_1$, $z_2$, $\tz_1$, $\tz_2$ defined above, we have the
following useful lemma.
\begin{lemma}\label{lemma1}
If
$|\nu|\le 1$,
then\, $\eqref{cond1} \Longrightarrow \eqref{cond2}$.
\end{lemma} 

\begin{proof}
Assume \eqref{cond1}. 
Set $x= \lambda \Delta t$ and $y = (r+\lambda) \Delta t$, then $0\le  x\le y$ and
\begin{equation*}
w_0 = x \nu,~~z_1 = \tz_1-\tfrac{1}{2}y,~~z_2 = \tz_2-\tfrac{1}{2}y.
\end{equation*}
It is obvious that $\Re z_1\le -\tfrac{1}{2}|w_0|$ and $\Re z_2\le -\tfrac{1}{2}|w_0|$.
Next,
\[
(|z_0|+|w_0|)^2
\le (|z_0|+x)^2 \le (|z_0|+y)^2 \le 4\Re \tz_1\Re \tz_2 + 4y\sqrt{\Re \tz_1\Re \tz_2} +y^2.
\]
Using that
\[
2\sqrt{\Re \tz_1\Re \tz_2} \le -(\Re \tz_1 + \Re \tz_2)
\]
there follows
\[
(|z_0|+|w_0|)^2
\le 4\Re \tz_1\Re \tz_2 - 2y (\Re \tz_1 + \Re \tz_2) + y^2 = 4\Re z_1\Re z_2
\]
and thus \eqref{cond2} is fulfilled.
\end{proof}

Write $\tz_0 = z_0 + w_0$.
For the adaptation \eqref{MCS1} it clearly holds that 
\begin{equation*}
R(w_0,z_0,z_1,z_2) = R(0,\tz_0,z_1,z_2).
\end{equation*}
In the literature on the stability analysis of the MCS scheme pertinent to PDEs with mixed derivative 
terms, a variety of favourable results has been derived on the stability bound 
$|R(0,\tz_0,z_1,z_2)|\le 1$ under the condition 
\begin{equation}\label{cond3}
|\tz_0|\le 2\sqrt{\Re z_1\Re z_2}
~~{\rm and}~~
\Re z_1\le 0,~\Re z_2\le 0.
\end{equation}
Obviously, \eqref{cond3} is weaker than \eqref{cond2}.
Hence, by virtue of Lemma \ref{lemma1}, one directly arrives at positive stability results for 
the adaptation \eqref{MCS1} pertinent to two-dimensional PIDEs with mixed derivative term.
The following theorem gives two main results, which are relevant to diffusion-dominated and
convection-dominated problems, respectively. 
\begin{theorem}\label{theorem1}
For the adaptation \eqref{MCS1} there holds:
\begin{itemize}
\item[(a)] If $\theta\ge \frac{1}{3}$, then $|R(w_0,z_0,z_1,z_2)|\le 1$ whenever $w_0, z_0, z_1, z_2 \in \R$
satisfy \eqref{cond2}.
\item[(b)] If $\frac{1}{2}\le \theta \le 1$, then $|R(w_0,z_0,z_1,z_2)|\le 1$ whenever $w_0, z_0, z_1, z_2 \in \C$
satisfy \eqref{cond2}.
\end{itemize}
\end{theorem}
\noindent
The above theorem guarantees unconditional contractivity (that is, without any restriction on 
the spatial mesh width or temporal step size) for all jump intensities $\lambda$.
It covers the most common values $\theta$ for the MCS scheme employed in the literature, 
namely $\theta= \frac{1}{3}$ and $\theta= \frac{1}{2}$.
Parts (a) and (b) of Theorem \ref{theorem1} stem from \cite[Thm.~2.5]{intHout09b} and 
\cite[Thm.~2.7]{intHout11}, respectively.
For various extensions and refinements of these stability results, we refer to 
\cite{intHout11,intHout13,intHout09b,Mishra16}.

The interesting question arises whether similar positive results as Theorem \ref{theorem1} 
are valid for the adaptations \eqref{MCS2} and \eqref{MCS3}.
The first result is negative:
\begin{theorem}\label{theorem2}
For the adaptation \eqref{MCS2} there holds:
\begin{itemize}
\item[(a)] If $|S(w_0,z_0,z_1,z_2)|\le 1$ whenever $w_0, z_0, z_1, z_2 \in \R$
satisfy \eqref{cond2} and $w_0\ge 0$, then $\theta\ge \frac{1}{2}\sqrt{2}$.
\item[(b)] There exists no $\theta$ such that $|S(w_0,z_0,z_1,z_2)|\le 1$ 
for all $w_0, z_0, z_1, z_2 \in \R$ satisfying \eqref{cond2}.
\end{itemize}
\end{theorem}
\begin{proof}
(a) Let $w_0=x \ge 0$ and consider $z_0=0$, $z_1=-\frac{1}{2}x-\xi$, $z_2=-\frac{1}{2}x$ with $\xi\ge 0$.
Then \eqref{cond2} is fulfilled and
\begin{equation*}
S(w_0,z_0,z_1,z_2) = 1 - \frac{( 1+\tfrac{1}{2}x) \xi}{(1+\tfrac{\theta}{2}x+\theta\xi)(1+\tfrac{\theta}{2}x)}
\left[1-\frac{(\tfrac{1}{2}-\theta)(x+\xi)}{(1+\tfrac{\theta}{2}x+\theta\xi)(1+\tfrac{\theta}{2}x)}\right].
\end{equation*}
Letting successively $\xi\rightarrow \infty$ and $x\rightarrow \infty$, the right-hand side tends to 
the value $1-1/\theta^2$. 
The modulus of this must be bounded from above by $1$, which yields $\theta\ge \frac{1}{2}\sqrt{2}$.

(b) Taking $w_0=-x\le 0$, an analogous derivation as in (a) leads to the requirement that $1+1/\theta^2$ 
must be bounded from above by $1$, which obviously does not hold.
\end{proof}
\noindent
Clearly, in Theorem \ref{theorem2} for the adaptation \eqref{MCS2} the popular values 
$\theta= \frac{1}{3}$ and $\theta= \frac{1}{2}$ are excluded.
We remark that the condition $w_0\in\R$, $w_0\ge 0$ is relevant to the case where the
eigenvalues $\nu$ of $\J$ are real and lie in the interval $[0,1]$.
This is often seen to be numerically fulfilled in our applications, 
compare Section~\ref{ExampleSec}.

For the adaptation \eqref{MCS3}, stability is determined by power-boundedness of the
companion matrix
\begin{equation*}
C(w_0,z_0,z_1,z_2)
= 
\begin{pmatrix}
\, T_1(w_0,z_0,z_1,z_2) & T_0(w_0,z_0,z_1,z_2)\, \\
1 & 0
\end{pmatrix}.
\end{equation*}
This is established by studying the roots of the characteristic polynomial
\begin{equation*}
P(\zeta;w_0,z_0,z_1,z_2) = \zeta^2 - T_1(w_0,z_0,z_1,z_2)\zeta - T_0(w_0,z_0,z_1,z_2).
\end{equation*}
For any given point $(w_0,z_0,z_1,z_2)$ the recurrence relation \eqref{rT} is stable if and only if the 
root condition holds: both roots $\zeta$ of $P$ have a modulus of at most one, and those with modulus 
equal to one are simple.
If $T_0=T_0(w_0,z_0,z_1,z_2)$ and $T_1=T_1(w_0,z_0,z_1,z_2)$ are real, then by the Schur criterion
the root condition is satisfied if and only if 
\begin{equation}\label{Schur}
|T_0| \le 1 ~~{\rm and}~~ T_0 + |T_1| \le 1
\end{equation}
and there is no multiple root with modulus equal to one.
\begin{theorem}\label{theorem3}
For the adaptation \eqref{MCS3} there holds:
\begin{itemize}
\item[(a)] If $\theta\ge \frac{1}{3}$, then \eqref{rT} is stable whenever $w_0, z_0, z_1, z_2 \in \R$
satisfy \eqref{cond2}  and $w_0\ge 0$.
\item[(b)] If \eqref{rT} is stable whenever $w_0, z_0, z_1, z_2 \in \R$ satisfy \eqref{cond2}, then
$\theta \ge (9+\sqrt{33})/16$.
\end{itemize}
\end{theorem}
\begin{proof}
(a) Write $R_0 = R(0,z_0,z_1,z_2)$.
Condition \eqref{cond2} implies
\begin{equation*}\label{cond4}
|z_0|\le 2\sqrt{z_1 z_2}
~~{\rm and}~~
z_1\le 0,~ z_2\le 0
\end{equation*}
and by \cite[Thm.~2.5]{intHout09b} it holds that $-1\le R_0 \le 1$.

If $w_0=0$, then $T_0=0$ and $T_1=R_0$ and the root condition is clearly fulfilled.
Assume next  $w_0 > 0$.
Using \eqref{cond2} yields
\begin{equation*}
z = w_0 + z_0 + z_1 + z_2  \le 2\sqrt{z_1 z_2} + z_1 + z_2 = - \left( \sqrt{-z_1} - \sqrt{-z_2} \right)^2 \le 0
\end{equation*}
and thus
\begin{equation*}
0<w_0 \le - (z_0 + z_1 + z_2) = |z_0 + z_1 + z_2|.
\end{equation*}
We have
\begin{equation}\label{Ts}
T_0 = - \frac{\tfrac{1}{2}w_0}{z_0 + z_1 + z_2} (R_0-1)
~~{\rm and}~~ T_1 = 1 + \frac{z+\tfrac{1}{2}w_0}{z_0 + z_1 + z_2} (R_0-1).
\end{equation}
Hence, 
\begin{equation*}
|T_0| \le \tfrac{1}{2} \left| R_0 - 1 \right| \le 1,
\end{equation*}
which proves the first condition in \eqref{Schur}.
Next,
\begin{equation*}
T_0 + T_1 = 1 + \frac{z}{z_0 + z_1 + z_2} (R_0-1) \le 1,
\end{equation*}
since $z_0 + z_1 + z_2<0$, $z\le 0$ and $R_0\le 1$.
Subsequently,
\begin{equation*}
T_0 - T_1 = -1 - \frac{z+w_0}{z_0 + z_1 + z_2} (R_0-1) = -R_0 - \frac{2w_0}{z_0 + z_1 + z_2} (R_0-1)\le 1,
\end{equation*}
since $w_0>0$, $z_0 + z_1 + z_2<0$ and $-1\le R_0 \le 1$.
This proves the second condition in \eqref{Schur}.

Finally, suppose $\zeta_0$ is a double root of $P$ with modulus one.
Since the coefficients of $P$ are real, it must hold that $\zeta_0 \in \{-1,1\}$
and consequently $T_0=-1$ and $T_1 \in \{-2,2\}$.
Using the expressions \eqref{Ts}, the case $(T_0,T_1)=(-1,-2)$ leads to
$R_0=-5$, which is a contradiction.
Next, the case $(T_0,T_1)=(-1,2)$ implies $z=0$ and $R_0=-1$, which is 
also a contradiction, cf. \eqref{R}.
Thus there is no multiple root with modulus one. 

(b) Let $w_0 = -2/\theta$, $z_0=0$, $z_1=z_2=-1/\theta$.
Then \eqref{cond2} is fulfilled and the necessary condition
\begin{equation*}
T_0 - T_1 = -R_0 - \frac{2w_0}{z_0 + z_1 + z_2} (R_0-1) = -R_0 - 2(R_0-1) \le 1
\end{equation*}
becomes $R_0 \ge \tfrac{1}{3}$.
There holds
\begin{equation*}
R_0 = 1-\frac{1}{2\theta} + \left(\frac{1}{2}-\theta\right) \frac{1}{4\theta^2}\,.
\end{equation*}
A simple calculation leads to the requirement that $8\theta^2-9\theta+\tfrac{3}{2} \ge 0$,
which is readily seen to imply $\theta \ge (9+\sqrt{33})/16$.
\end{proof}
\noindent
Part (a) of Theorem \ref{theorem3} for the adaptation \eqref{MCS3} is a positive result 
and includes the common values $\theta= \frac{1}{3}\,,\, \frac{1}{2}$.
Part (b) of the theorem is not favourable, on the other hand, given the 
(large) lower bound $(9+\sqrt{33})/16 \approx 0.92$.\\

In the following we derive positive conclusions on the stability of both adaptations 
\eqref{MCS2}, \eqref{MCS3} under the {\it assumption that $\lambda T$ is of moderate size}.
These are obtained by a perturbation-type argument with respect to $w_0$, as employed 
in \cite{Kaushansky17}.
Define
\begin{equation}
Q(z_0,z_1,z_2) = \frac{1}{p}+\theta\frac{z_0}{p^2}+(\tfrac{1}{2}-\theta)\frac{z_0+z_1+z_2}{p^2}\,,
\end{equation}
and consider the condition
\begin{equation}\label{cond5}
|z_0|\le 2\sqrt{\Re z_1\Re z_2}
~~{\rm and}~~
\Re z_1\le 0,~\Re z_2\le 0,
\end{equation}
which is clearly valid under (\ref{cond1}).
\begin{lemma}\label{lemma2}
Let $L = \max\{1/\theta\,,\,2\}$. Then $|Q(z_0,z_1,z_2)| \le L$ whenever \eqref{cond5} holds.
\end{lemma}
\begin{proof}
It is easily seen that $|p|\ge 1$ and $|z_0| \le |z_1+z_2|$. Hence,
\begin{equation*}
|Q(z_0,z_1,z_2)| \le
\left| 1+\frac{z_0}{2p}+\left(\tfrac{1}{2}-\theta\right)\frac{z_1+z_2}{p} \right| \le
1+\left( \tfrac{1}{2} + \left|\tfrac{1}{2}-\theta\right| \right) \left| \frac{z_1+z_2}{p} \right|.
\end{equation*}
Next, it can be verified that $|p|\ge \theta |z_1+z_2|$ 
and therefore
\begin{equation*}
|Q(z_0,z_1,z_2)| \le
1+\left( \tfrac{1}{2} + \left|\tfrac{1}{2}-\theta\right| \right)\! /\theta = \max\{1/\theta\,,\,2\}.
\end{equation*}
\end{proof}

\begin{theorem}\label{theorem2b}
Let $M = e^{(1+L)\lambda T}$ and $n\Delta t \le T$.
For the adaptation \eqref{MCS2} there holds:
\begin{itemize}
\item[(a)] If $\theta\ge \frac{1}{3}$, then $|S(w_0,z_0,z_1,z_2)^n|\le M$ 
whenever $|w_0| \le \lambda \Delta t$~and $z_0, z_1, z_2 \in \R$ satisfy \eqref{cond5}.
\item[(b)] If $\frac{1}{2}\le \theta \le 1$, then $|S(w_0,z_0,z_1,z_2)^n|\le M$ 
whenever $|w_0| \le\lambda \Delta t$ and $z_0, z_1, z_2 \in \C$ satisfy \eqref{cond5}.
\end{itemize}
\end{theorem}
\begin{proof}
Write $R_0 = R(0,z_0,z_1,z_2)$ as before.
Upon inserting $z = w_0+(z_0+z_1+z_2)$ into (\ref{S}), it is readily shown that 
$S=S(w_0,z_0,z_1,z_2)$ can be expressed as
\begin{equation*}
S = R_0 + \left[ \tfrac{1}{2}(R_0-1) +Q \right]\! w_0 + \tfrac{1}{2} Q w_0^2.
\end{equation*}
Assume $|R_0|\le 1$. Then, by applying Lemma \ref{lemma2}, 
\begin{equation*}
|S| \le 1 + (1+L) |w_0| + \tfrac{1}{2} L |w_0|^2 \le e^{(1+L) |w_0|} \le e^{(1+L) |w_0|} \le e^{(1+L) \lambda \Delta t}.
\end{equation*}
Consequently, for $n\Delta t \le T$,
\begin{equation*}
|S^n| \le e^{(1+L) \lambda T} = M.
\end{equation*}
Parts (a) and (b) now follow by virtue of \cite[Thm.~2.5]{intHout09b} and 
\cite[Thm.~2.7]{intHout11}, respectively.
\end{proof}

\noindent
Denote by $\| \cdot \|$ the maximum norm for matrices.
\begin{theorem}\label{theorem3b}
Let $M = e^{2L \lambda T}$ and $n\Delta t \le T$.
For the adaptation \eqref{MCS3} there holds:
\begin{itemize}
\item[(a)] If $\theta\ge \frac{1}{3}$, then $\|C(w_0,z_0,z_1,z_2)^n\|\le M$ 
whenever $|w_0| \le \lambda \Delta t$~and $z_0, z_1, z_2 \in \R$ satisfy \eqref{cond5}.
\item[(b)] If $\frac{1}{2}\le \theta \le 1$, then $\|C(w_0,z_0,z_1,z_2)^n\|\le M$ 
whenever $|w_0| \le\lambda \Delta t$ and \mbox{$z_0, z_1, z_2\in \C$} satisfy \eqref{cond5}.
\end{itemize}
\end{theorem}
\begin{proof}
From (\ref{T1}), (\ref{T2}) it is directly clear that
\begin{equation*}
C =   
C(w_0,z_0,z_1,z_2) =
\begin{pmatrix}
R_0 & 0\, \\
1 & 0
\end{pmatrix}
+
\begin{pmatrix}
\tfrac{3}{2}Q w_0 & -\tfrac{1}{2}Q w_0 \\
0 & 0
\end{pmatrix}.
\end{equation*}
Assume $|R_0|\le 1$. Then, by Lemma \ref{lemma2}, 
\begin{equation*}
\| C \| \le 1+2L |w_0| \le e^{2L |w_0|} \le e^{ 2L \lambda \Delta t}
\end{equation*}
and the remainder of the proof is the same as that of Theorem \ref{theorem2b}.
\end{proof}

\section{Numerical examples}\label{ExampleSec}
In this section we examine by numerical experiments the stability and convergence behaviour
of the three adaptations \eqref{MCS1}, \eqref{MCS2}, \eqref{MCS3} of the MCS scheme in the 
application to the semidiscretized Bates PIDE described in Section \ref{BatesSec}.
We consider the four parameter sets for the Bates model and European put option listed in
Table~\ref{table1}.
Case~I has  been considered in e.g. \cite{Salmi14a,vonSydow15} and 
Case~II in \cite{Ballestra16,Ballestra10}.
Case~III stems from the recent monograph \cite{Itkin17}.
Case~IV is new and has been constructed as a specifically challenging example, 
where in particular $\lambda T$ is large and the Feller condition 
$2\kappa\eta>\sigma^2$ is violated.

\begin{table}
\begin{center}
\begin{tabular}{|c|r|r|r|r|r|}
\hline
& {\rm Case I} & {\rm Case II} & {\rm Case III} & {\rm Case IV}\\
\hline
  $\kappa$   &2     &2      &1.5   &2.5   \\
  $\eta$     &0.04  &0.04   &0.1   &0.05  \\
  $\sigma$   &0.25  &0.4    &0.3   &0.6   \\
  $\rho$     &-0.5  &-0.5   &-0.5  &-0.8  \\
  $r$        &0.03  &0.03   &0.05  &0.01  \\
  $\lambda$  &0.2   &5      &5     &10    \\
  $\gamma$   &-0.5  &-0.005 &0.3   &-0.05 \\
  $\delta$   &0.4   &0.1    &0.1   &0.01  \\
  $T$        &0.5   &0.5    &1     &5     \\
  $K$        &100   &100    &100   &100   \\
  \hline
\end{tabular}
\vskip0.4cm
\caption{Parameter sets for the Bates model and European put option.}
\label{table1}
\end{center}
\end{table}

For the MCS scheme the common parameter values $\theta=\frac{1}{3}$ and 
$\theta=\frac{1}{2}$ are selected.
Recall that the latter choice yields the Craig--Sneyd (CS) scheme.
In addition, we study the reduced version of each of \eqref{MCS1}, \eqref{MCS2}, 
\eqref{MCS3} that ends directly with the computation of $Y_2$ and sets $U_n=Y_2$.
These can be viewed as adaptations of the well-known Douglas (Do) scheme 
(cf.~e.g.~\cite{Hundsdorfer02,Hundsdorfer03,intHout07}) and will be 
applied with $\theta=\frac{1}{2}$.

We investigate the {\it global temporal discretization error}\, at time $t=T=N \Delta t$,
\begin{equation*}\label{t_error}
{\widehat e}\,(N; m_1, m_2) = \max \left\{\, |U_k(T)-U_{N,\,k}|:~\tfrac{1}{2} K 
< s_i < \tfrac{3}{2}K,~0 < v_j < 1 \right\},
\end{equation*}
where index $k$ is such that $U_k(T)$ and $U_{N,\,k}$ correspond to the spatial 
grid point $(s_i,v_j)$.
Clearly, the temporal error under consideration is defined via the maximum norm.
The set $\{(s,v):\,\tfrac{1}{2} K < s < \tfrac{3}{2}K,\, 0 < v < 1 \}$ forms a 
natural region of interest for practical applications. 
For the truncation of the spatial domain, $S_{\max} = 8K$ and $V_{\max}=5$ are 
taken, and the semidiscrete solution vector $U(T)$ in 
${\widehat e}$\, is approximated to high accuracy by applying a suitable time 
stepping method with a very small step size.

Figure~\ref{Eigenvalues_J} shows the eigenvalues of the matrix $\J$ in the
four cases of Table~\ref{table1} when $m_1=2m_2=200$.
They all lie in or are close to the real interval $[0,1]$ in Cases 
\mbox{I, II, III}, whereas they are dispersed in the complex unit 
disk in Case IV.

Figure~\ref{TemporalErrors1} displays for Case I (left column) and Case II (right
column) the temporal errors ${\widehat e}\,(N;200,100)$ versus $1/N$ for a range 
of values $N$ with $10\le N\le 1000$.
Figure~\ref{TemporalErrors2} shows the same for Case III (left column) and 
Case IV (right column).
The $i$-th row of each figure concerns the $i$-th adaptation ($i=1,2,3$).
The results corresponding to the MCS scheme with $\theta=\frac{1}{3}$ and 
$\theta=\frac{1}{2}$ are indicated by green and blue circles, respectively, 
and those corresponding to the Do scheme with $\theta=\frac{1}{2}$ by red 
diamonds.

\begin{figure}
\begin{center}
\includegraphics[width=0.67\textwidth]{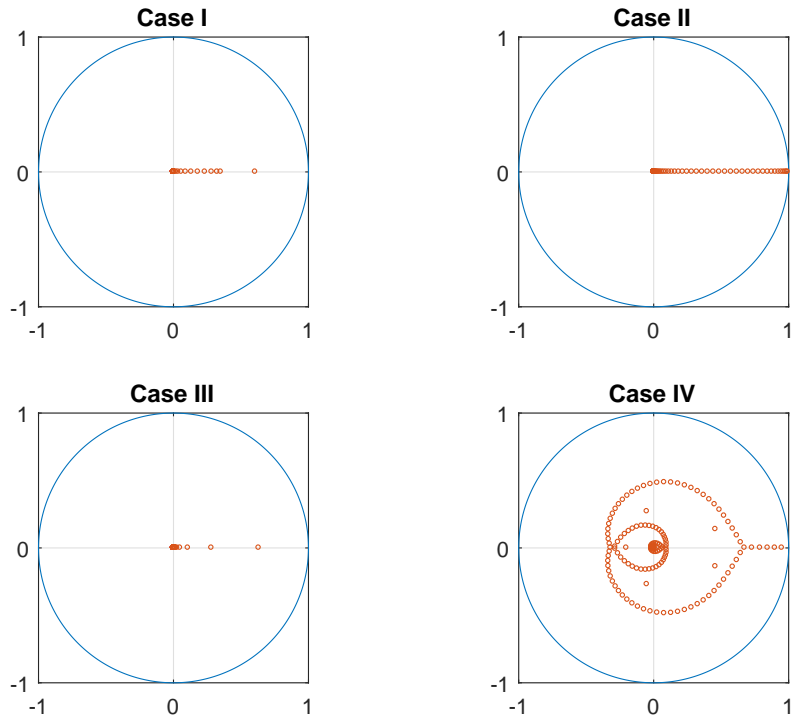}
\end{center}
\caption{Eigenvalues of the matrix $\J$ in the complex plane, indicated by red circles,
in the four cases of Table~\ref{table1}.}
\label{Eigenvalues_J} 
\end{figure}

In Figures \ref{TemporalErrors1}, \ref{TemporalErrors2} we observe that in all
but one instance the temporal errors always remain below a moderate value and 
are (essentially) monotonically decreasing as $N$ increases.
In Case IV the second adaptation yields large errors for 
each ADI scheme if $N\lesssim 50$.
We attribute this to a lack of unconditional stability, which may reveal
itself when $\lambda T$ is large.
Except for this instance, for each given adaptation, the MCS scheme with
$\theta=\frac{1}{3}$ always yields temporal errors that are smaller than, 
or approximately equal to, those for the CS and Do schemes, and it always 
shows a smooth, second-order convergence behaviour.
Additional numerical experiments on finer spatial grids, such as with
$m_1=2m_2=400$, show that the temporal errors for all adaptations of the 
MCS scheme with $\theta=\frac{1}{3}$ are visually identical, indicating
that second-order convergence holds in a favourable {\it stiff sense}.

In Cases I, II the temporal errors obtained with the CS and Do schemes are 
often relatively large for $N \lesssim 100$.
This is due to the nonsmoothness of the initial function and can be resolved 
by using backward Euler damping (Rannacher time stepping).
This is computationally expensive, however, and does not alter our 
above conclusion.

Finally, we note that Figures \ref{TemporalErrors1}, \ref{TemporalErrors2} 
suggest that the temporal error constant increases as $\lambda T$ increases. 
Additional numerical experiments, where $\lambda$ and $T$ are varied and
all else is kept fixed, confirm this.

\clearpage

\begin{figure}[h!]
\begin{center}
\begin{tabular}{c c}
         \includegraphics[width=0.5\textwidth]{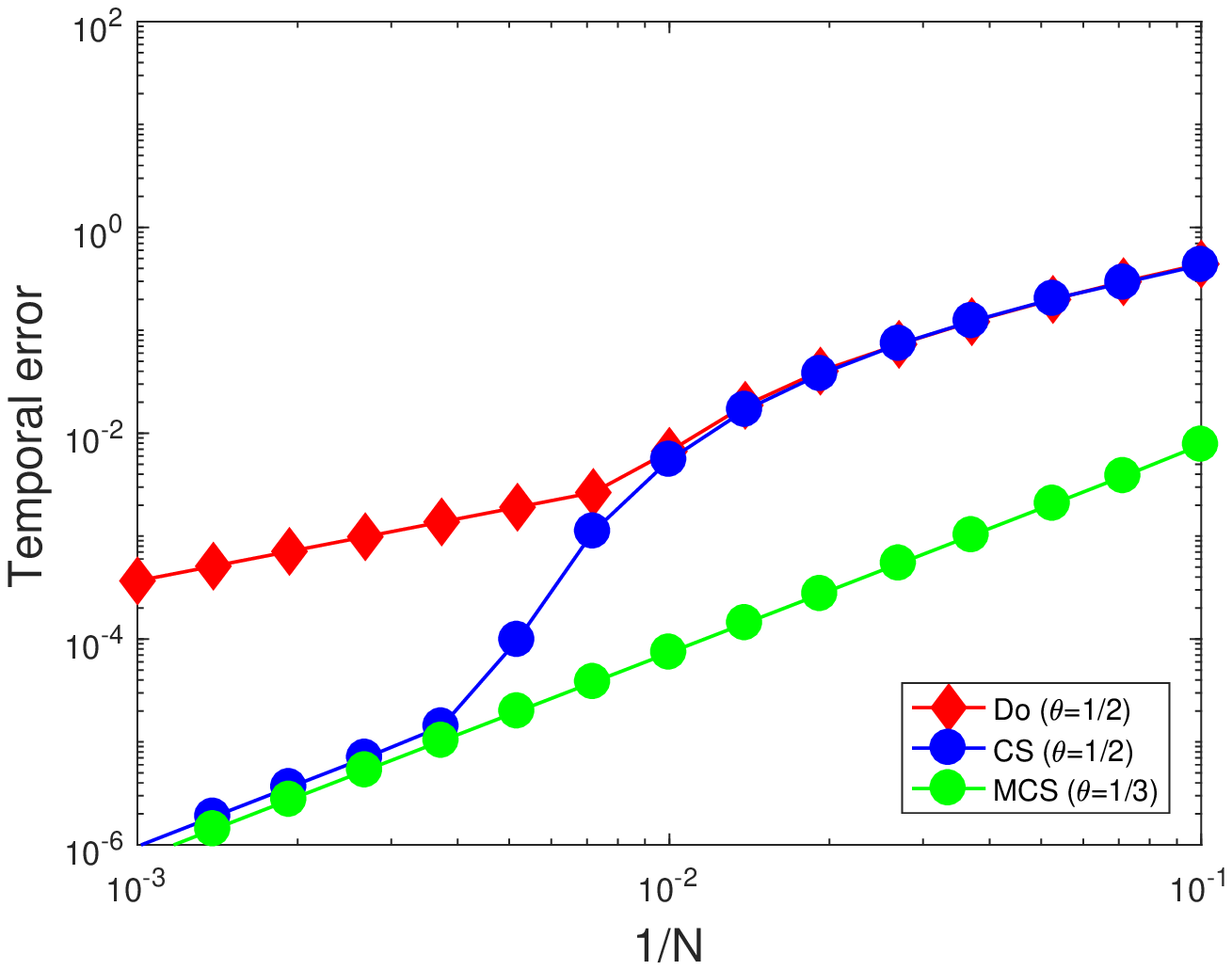}&
         \includegraphics[width=0.5\textwidth]{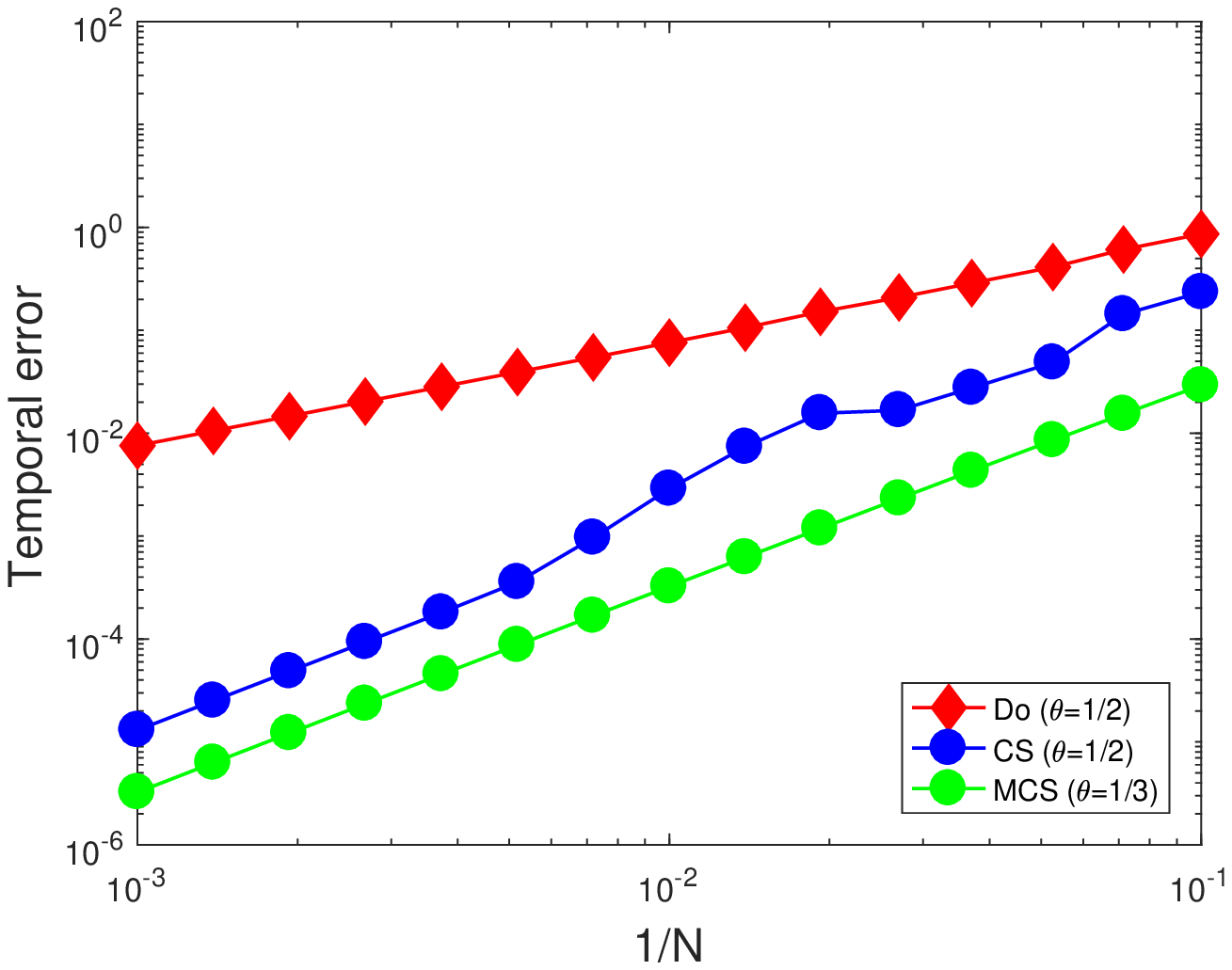}\\
         \includegraphics[width=0.5\textwidth]{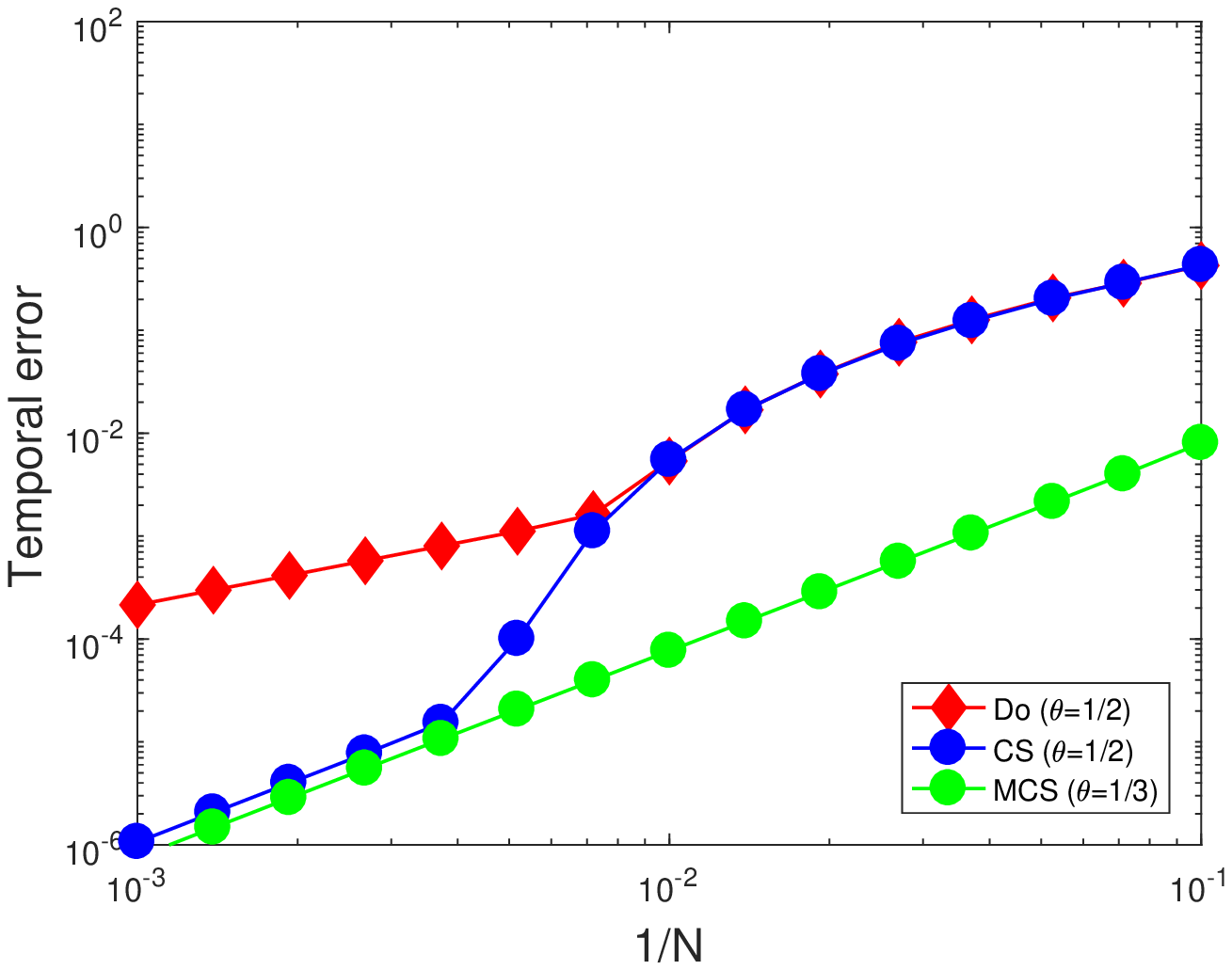}&
         \includegraphics[width=0.5\textwidth]{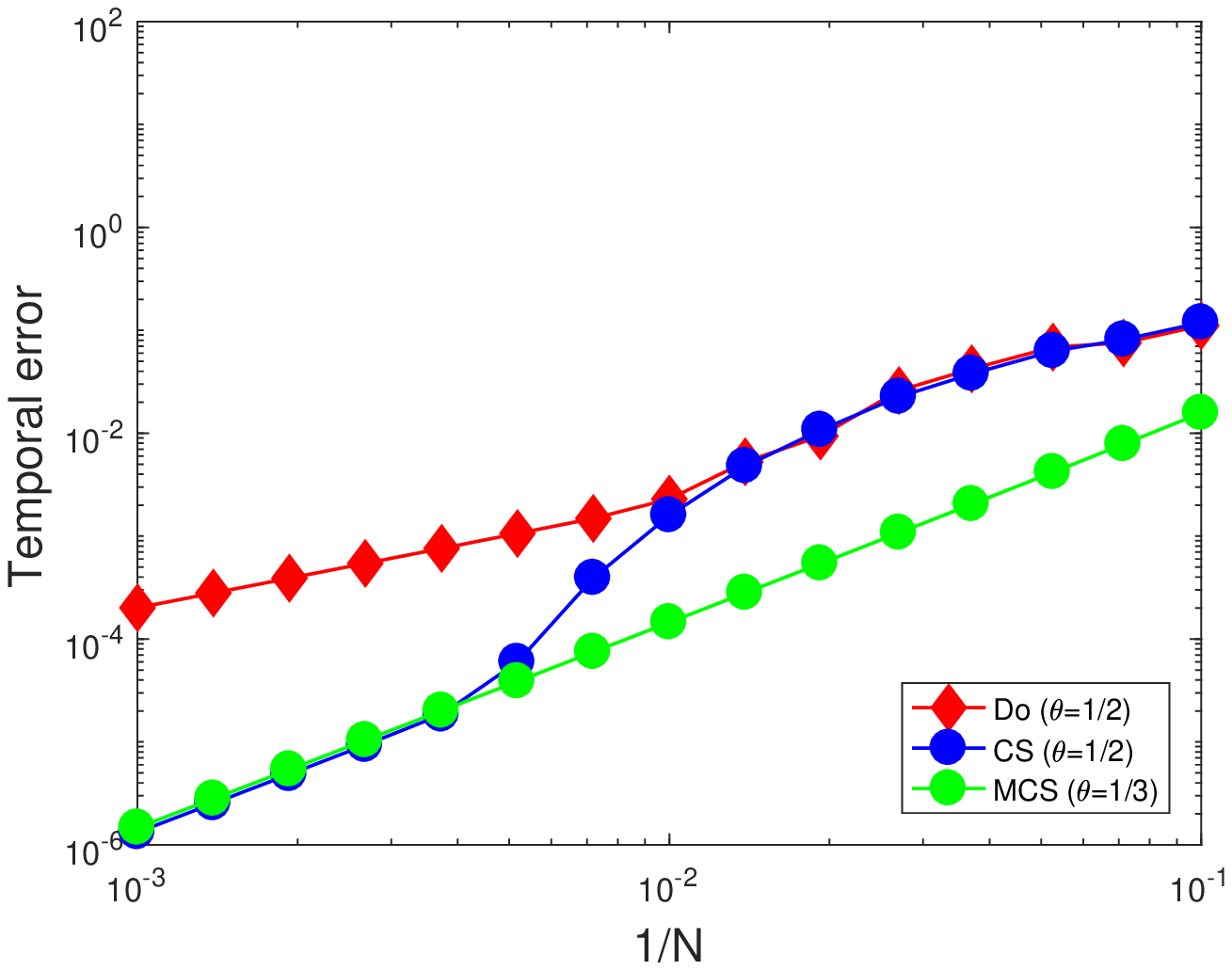}\\
         \includegraphics[width=0.5\textwidth]{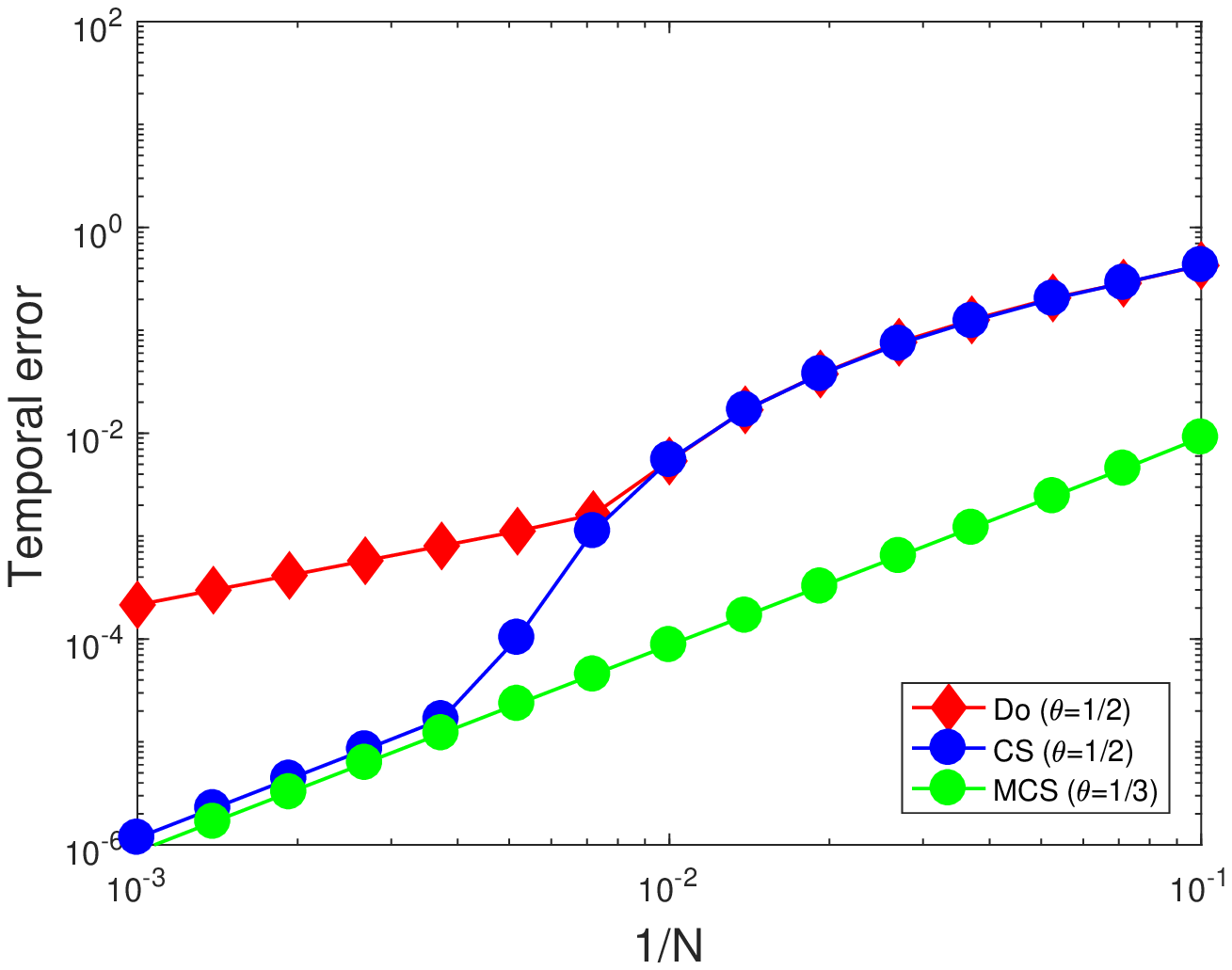}&
         \includegraphics[width=0.5\textwidth]{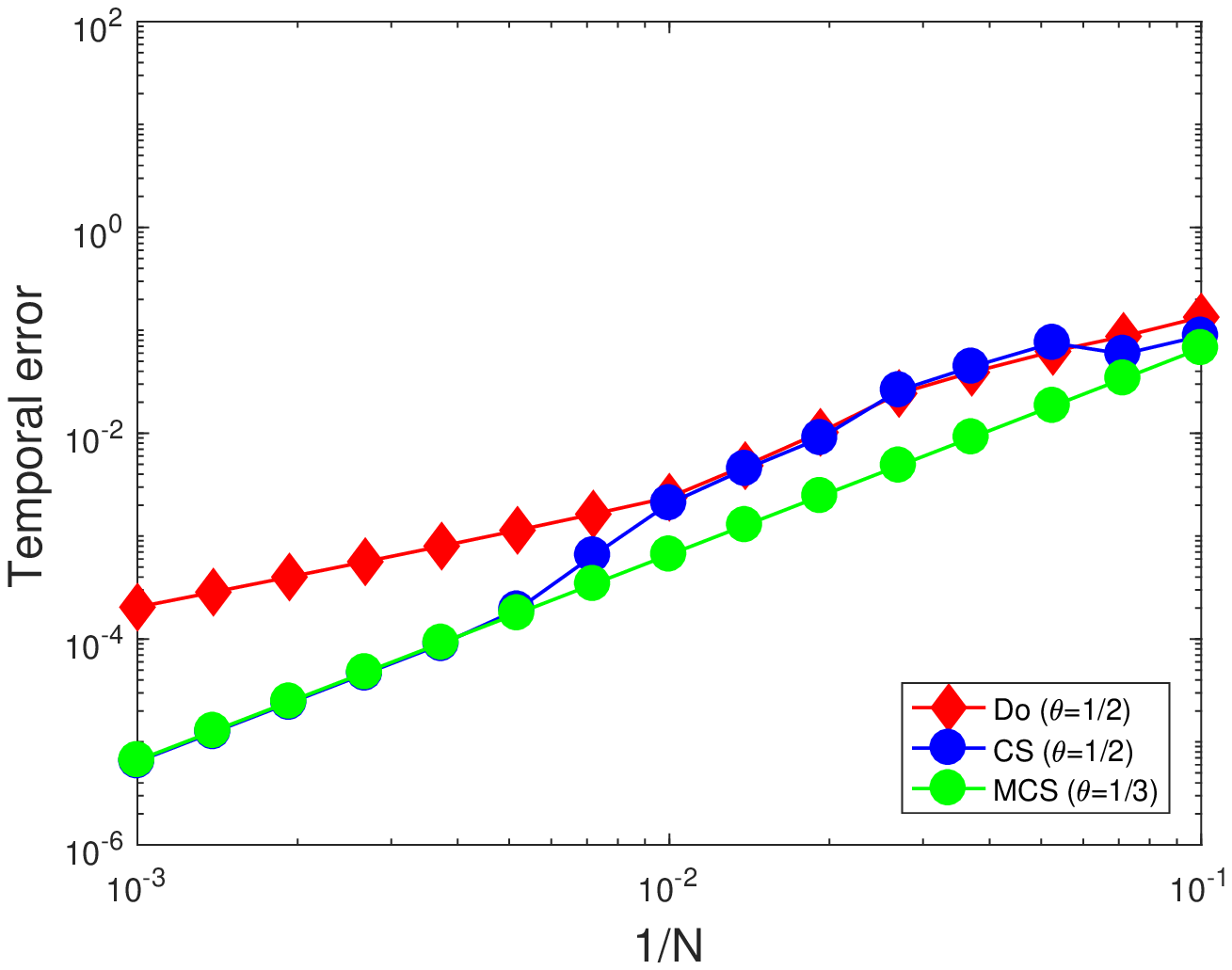}
\end{tabular}
\end{center}
\caption{Global temporal errors ${\widehat e}\,(N;200,100)$ versus $1/N$.
Case I: left column. 
Case II: right column.
Adaptation \eqref{MCS1}: top row.
Adaptation \eqref{MCS2}: middle row.
Adaptation \eqref{MCS3}: bottom row.
ADI schemes: 
MCS with $\theta=\frac{1}{3}$ (green circles),
MCS with $\theta=\frac{1}{2}$ (blue circles) and
Do with $\theta=\frac{1}{2}$ (red diamonds).
}
\label{TemporalErrors1}
\end{figure}
\clearpage

\begin{figure}[h!]
\begin{center}
\begin{tabular}{c c}
         \includegraphics[width=0.5\textwidth]{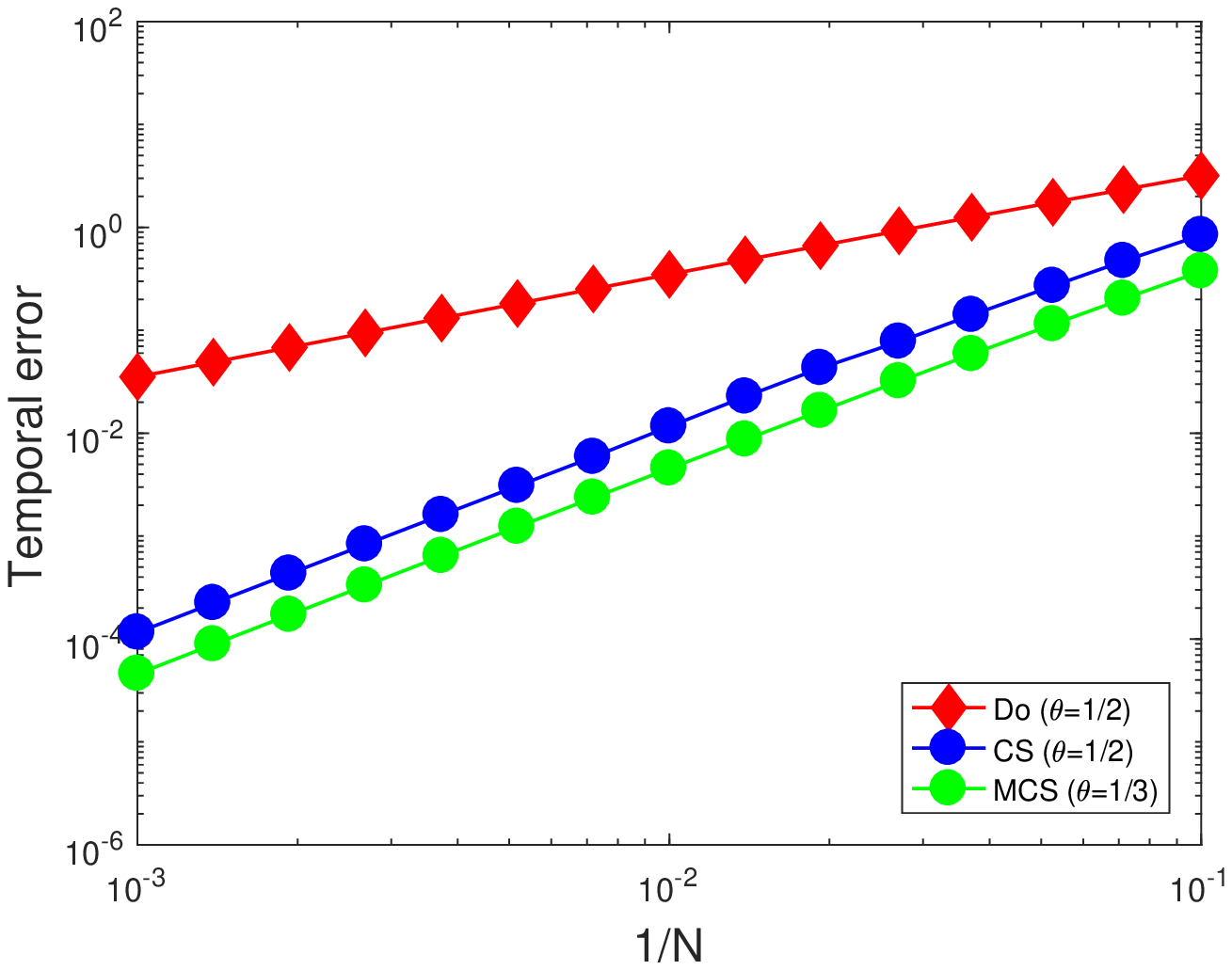}&
         \includegraphics[width=0.5\textwidth]{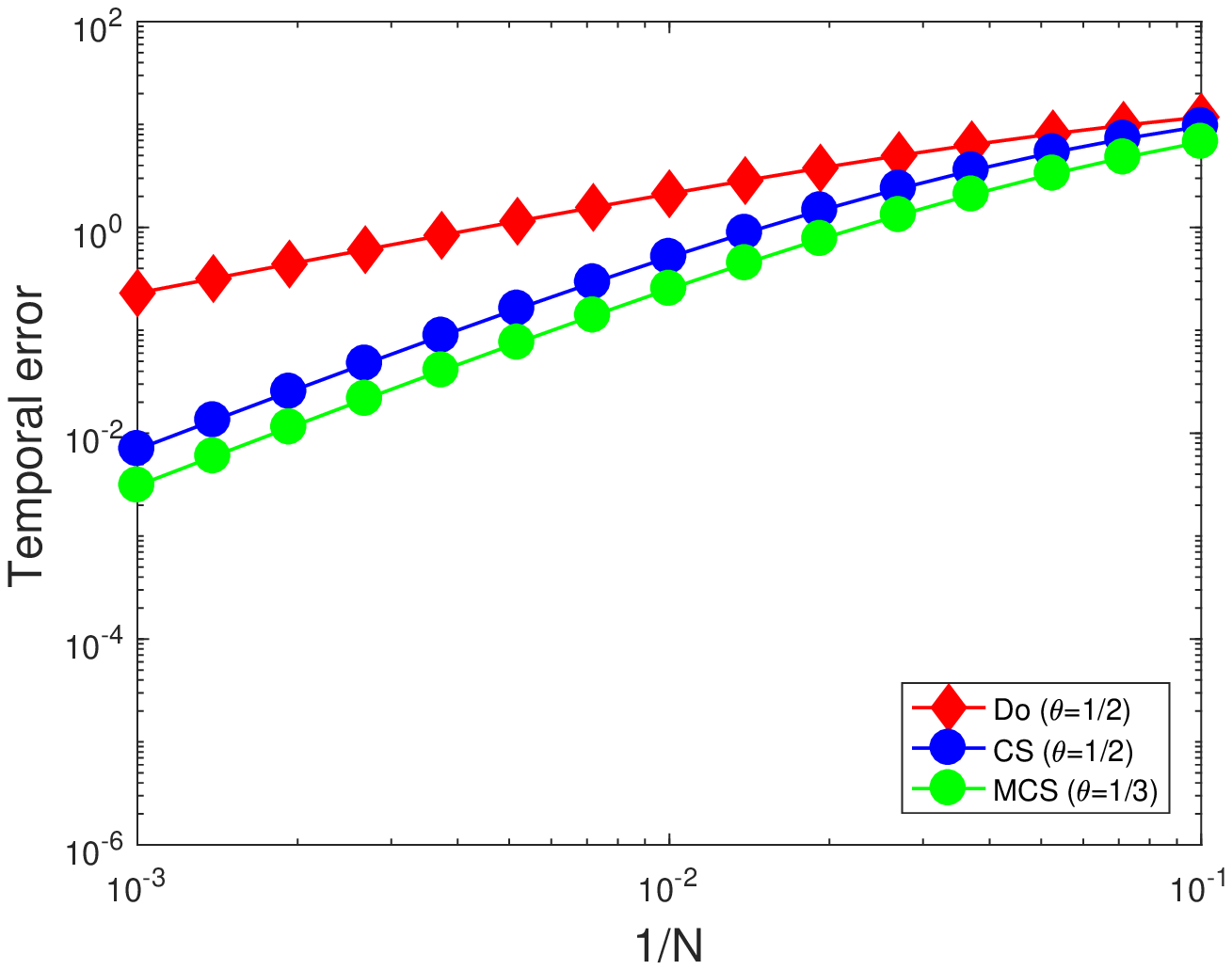}\\
         \includegraphics[width=0.5\textwidth]{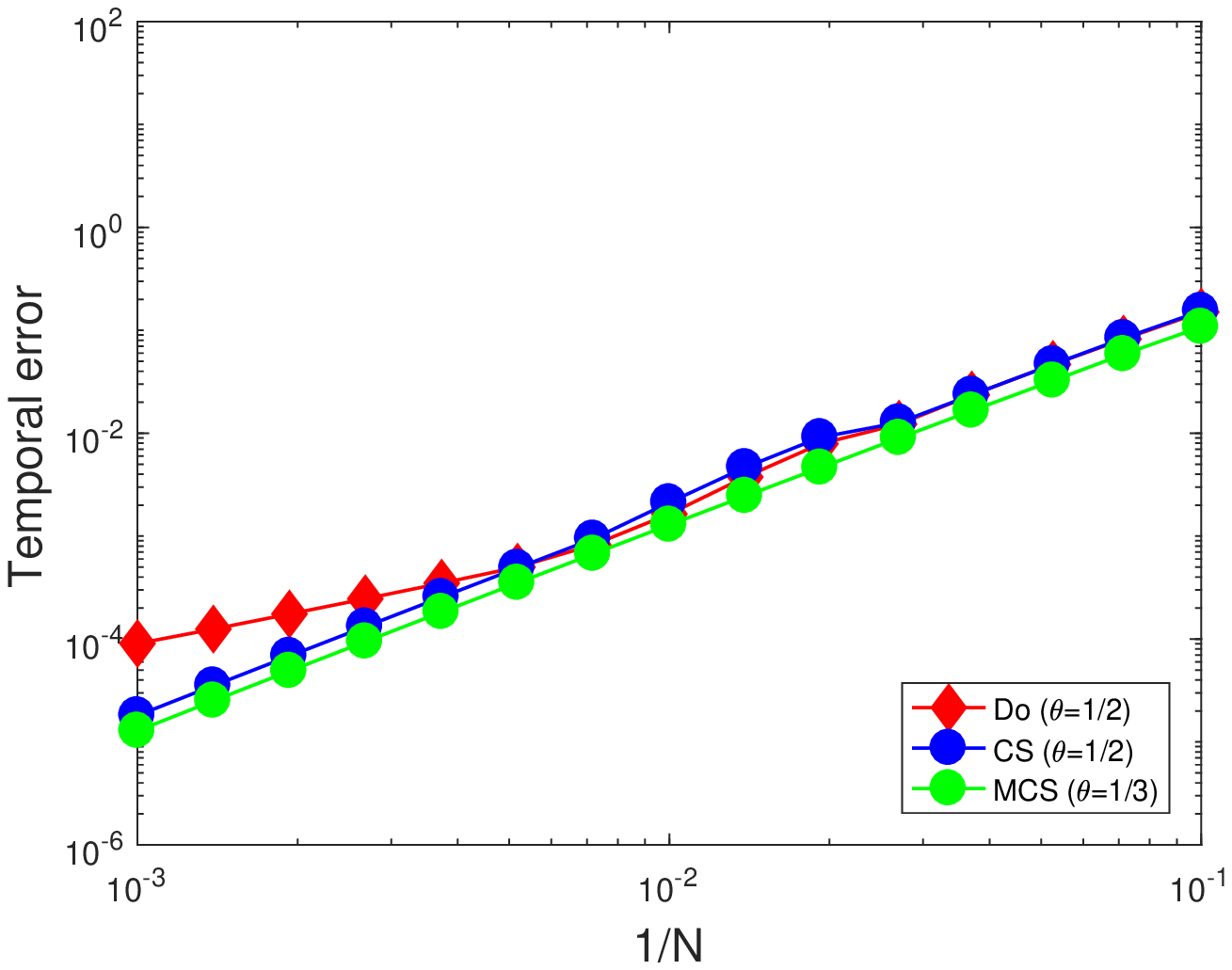}&
         \includegraphics[width=0.5\textwidth]{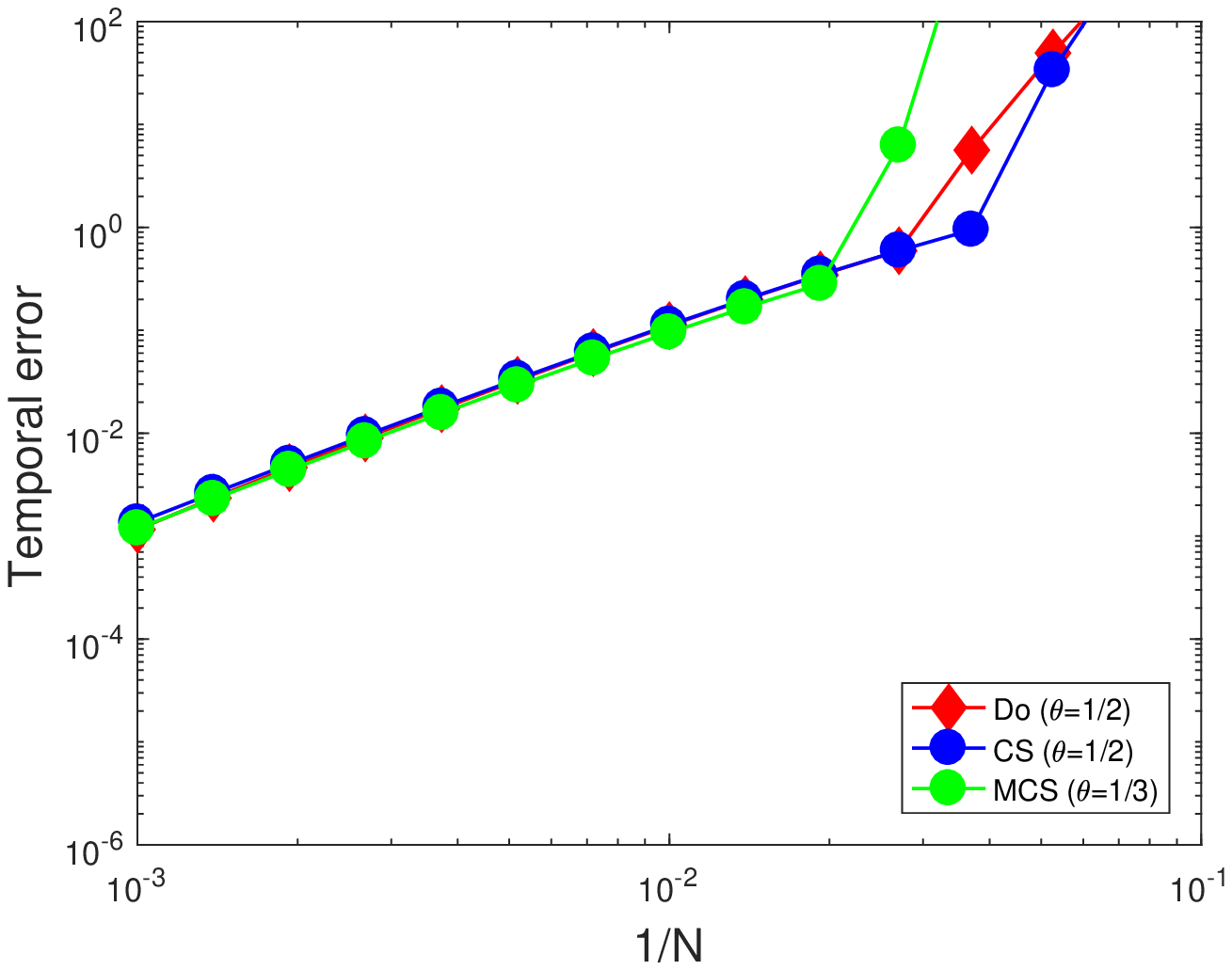}\\
         \includegraphics[width=0.5\textwidth]{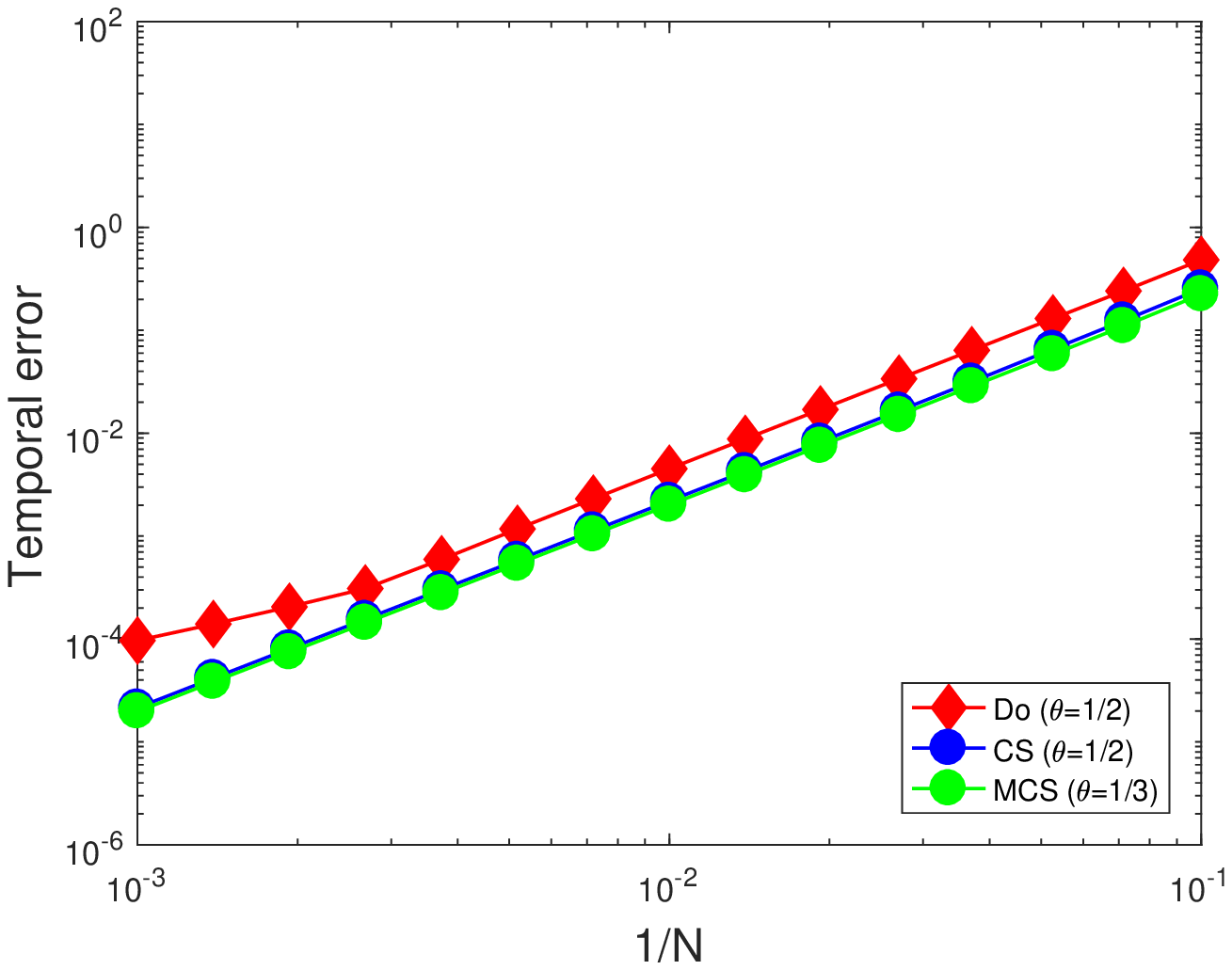}&
         \includegraphics[width=0.5\textwidth]{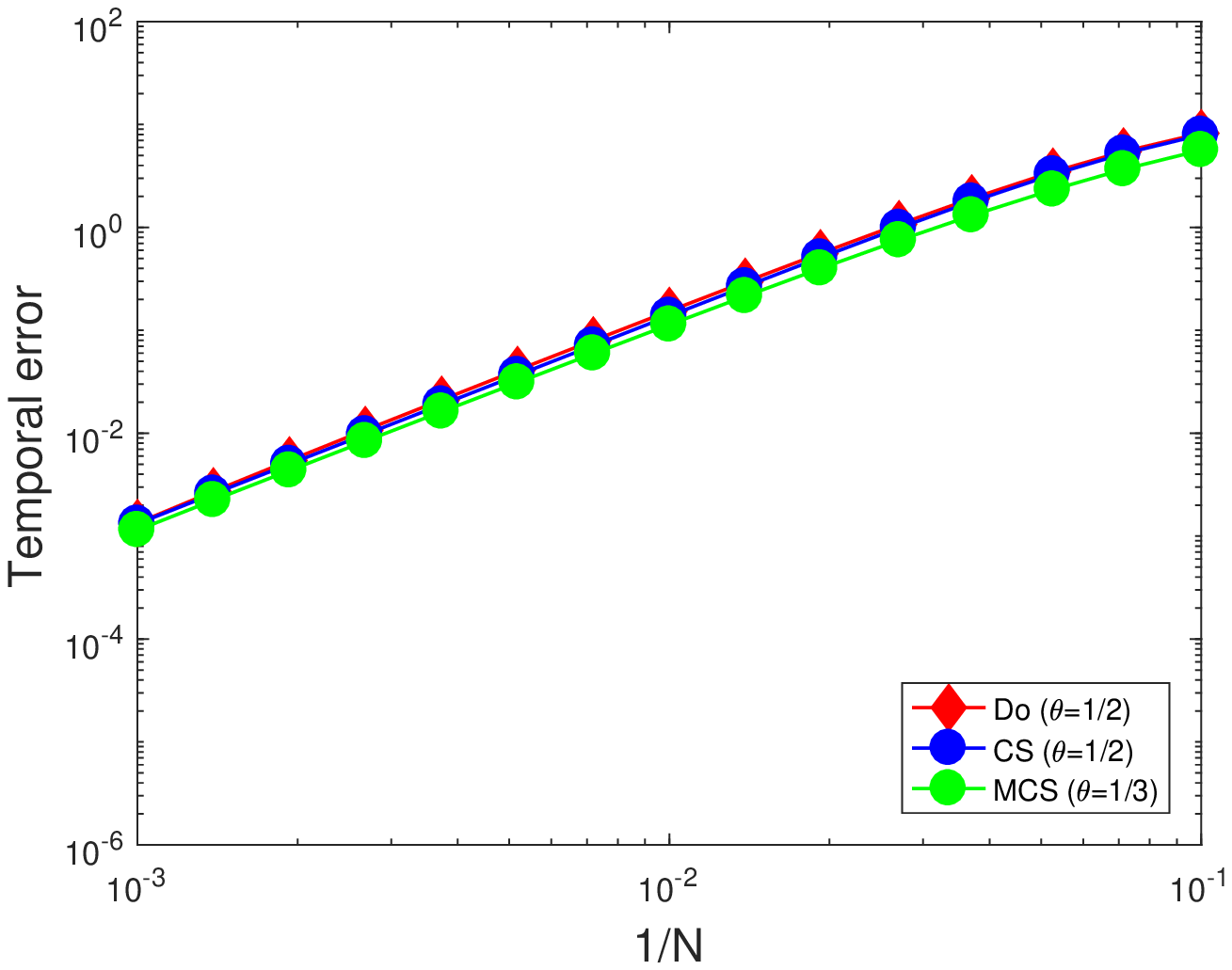}
\end{tabular}
\end{center}
\caption{Global temporal errors ${\widehat e}\,(N;200,100)$ versus $1/N$.
Case III: left column. 
Case IV: right column.
Adaptation \eqref{MCS1}: top row.
Adaptation \eqref{MCS2}: middle row.
Adaptation \eqref{MCS3}: bottom row.
ADI schemes: 
MCS with $\theta=\frac{1}{3}$ (green circles),
MCS with $\theta=\frac{1}{2}$ (blue circles) and
Do with $\theta=\frac{1}{2}$ (red diamonds).
}\label{TemporalErrors2}
\end{figure}
\clearpage

\section{Conclusion}\label{ConcSec}
In this paper we have studied three adaptations of the MCS scheme to 
two-dimensional PIDEs that are second-order consistent in the classical 
ODE sense.
They treat both the integral term and mixed derivative term explicitly 
and require per time step the solution of four linear systems with 
essentially tridiagonal matrices.
The first and second adaptations, \eqref{MCS1} and \eqref{MCS2}, handle the
integral term in a one-step fashion and employ two multiplications per time 
step with the dense matrix corresponding to this term.
The third adaptation, \eqref{MCS3}, deals with the integral term in a two-step 
fashion and uses just one multiplication per time step with the pertinent 
dense matrix.
In view of the stability results obtained in Section~\ref{ADISec} and  
numerical experiments for the Bates PIDE performed in Section~\ref{ExampleSec}, 
we recommend the third adaptation of the MCS scheme with parameter 
value $\theta=\frac{1}{3}$, 
directly followed by the first adaptation of this scheme.

\bibliographystyle{siam}
\bibliography{ADI_Bates}

\end{document}